\documentclass[12pt, twoside, leqno]{article}
\usepackage{amsmath,amsthm}
\usepackage{amssymb,latexsym}
\usepackage{enumerate}

\usepackage[mathscr]{eucal}

\newtheorem{thrm}{Theorem}[section]
\newtheorem{lem}[thrm]{Lemma}
\newtheorem{prop}[thrm]{Proposition}
\newtheorem{cor}[thrm]{Corollary}
\theoremstyle{definition}
\newtheorem{definition}[thrm]{Definition}
\newtheorem{remark}[thrm]{Remark}
\numberwithin{equation}{section}

\newtheorem{example}[thrm]{Example}

\newcommand{\labeq}[1]{\label{eq:#1}}
\newcommand{\refeq}[1]{(\ref{eq:#1})}
\newcommand{\labt}[1]{\label{thm:#1}}
\newcommand{\reft}[1]{Theorem~\ref{thm:#1}}
\newcommand{\labl}[1]{\label{lemma:#1}}
\newcommand{\refl}[1]{Lemma~\ref{lemma:#1}}
\newcommand{\labd}[1]{\label{definition:#1}}
\newcommand{\refd}[1]{Definition~\ref{definition:#1}}
\newcommand{\labc}[1]{\label{coro:#1}}
\newcommand{\refc}[1]{Corollary~\ref{coro:#1}}

\newcommand{\labp}[1]{\label{prop:#1}}
\newcommand{\refp}[1]{Proposition~\ref{prop:#1}}
\newcommand{\labe}[1]{\label{ex:#1}}

\newcommand{\e}{\epsilon}
\newcommand{\om}{\omega}
\newcommand{\fsz}{\frac {\partial f_i} {\partial z}}
\newcommand{\fsw}{\frac {\partial f_i} {\partial w}}
\newcommand{\qn}{q_1 q_2 \cdots q_n}
\newcommand{\qm}{q_1 q_2 \cdots q_m}
\newcommand{\formalsum}{\sum_{n=1}^{\infty} \frac {E_n} {q_1 q_2 \ldots q_n}}
\newcommand{\Fformalsum}{\sum_{n=1}^{\infty} \frac {E_{F,n}} {q_1 q_2 \ldots q_n}}
\newcommand{\dimh}[1]{\hbox{dim$_{\hbox{H}}$} #1}
\newcommand{\dimb}[1]{\hbox{dim}_{\hbox{B}} #1}
\newcommand{\ldimb}[1]{\underline{\hbox{dim}}_{\hbox{B}} #1}
\newcommand{\udimb}[1]{\overline{\hbox{dim}}_{\hbox{B}} #1}
\newcommand{\dimht}{\hbox{dim$_{\hbox{H}}$} \Theta_Q}
\newcommand{\dimbt}{\hbox{dim}_{\hbox{B}} \Theta_Q}
\newcommand{\ldimbt}{\underline{\hbox{dim}}_{\hbox{B}} \Theta_Q}
\newcommand{\udimbt}{\overline{\hbox{dim}}_{\hbox{B}} \Theta_Q}
\newcommand{\tq}{\Theta_Q}
\newcommand{\tqp}{\Theta_Q'}
\newcommand{\tqa}{\Theta_{Q_\al}}
\newcommand{\dimhtp}{\hbox{dim$_{\hbox{H}}$} \tqp}

\newcommand{\dimhtqa}{\hbox{dim$_{\hbox{H}}$} \tqa}
\newcommand{\dimhtqap}{\hbox{dim$_{\hbox{H}}$} \tqa'}

\newcommand{\ldimbtqa}{\underline{\hbox{dim}}_{\hbox{B}} \tqa}
\newcommand{\udimbtqa}{\overline{\hbox{dim}}_{\hbox{B}} \tqa}
\newcommand{\la}{\lambda}

\newcommand{\al}{\alpha}
\newcommand{\ga}{\gamma}
\newcommand{\be}{\beta}
\newcommand{\floor}[1]{\lfloor #1 \rfloor} 
\newcommand{\ceil}[1]{\lceil #1 \rceil}

\numberwithin{equation}{section}

\begin{document}


\baselineskip=17pt


\title{Cantor series constructions of sets of normal numbers}

\author{Bill Mance\\
Department of Mathematics, The Ohio State University\\
231 West 18th Avenue\\
Columbus, OH 43210-1174\\
E-mail: mance@math.ohio-state.edu}
\date{}

\maketitle


\renewcommand{\thefootnote}{}



\renewcommand{\thefootnote}{\arabic{footnote}}
\setcounter{footnote}{0}


\begin{abstract}
Let $Q=(q_n)_{n=1}^{\infty}$ be a sequence of integers greater than or equal to $2$. We say that a real number $x$ in $[0,1)$ is {\it $Q$-distribution normal} if the sequence $(q_1q_2 \cdots q_n x)_{n=1}^{\infty}$ is uniformly distributed mod $1$. In \cite{Lafer}, P. Lafer asked for a construction of a $Q$-distribution normal number for an arbitrary $Q$. Under a mild condition on $Q$, we construct a set $\Theta_Q$ of $Q$-distribution normal numbers.  This set is perfect and nowhere dense.  Additionally, given any $\alpha$ in $[0,1]$, we provide an explicit example of a sequence $Q$ such that the Hausdorff dimension of $\Theta_Q$ is equal to $\alpha$.  Under a certain growth condition on $q_n$, we provide a discrepancy estimate that holds for every $x$ in $\Theta_Q$.

\end{abstract}

\section{Introduction}

\begin{definition}\labd{1.1} Let $b$ and $k$ be positive integers.  A {\it block of length $k$ in base $b$} is an ordered $k$-tuple of integers in $\{0,1,\ldots,b-1\}$.  A {\it block of length $k$} is a block of length $k$ in some base $b$.  A {\it block} is a block of length $k$ in base $b$ for some integers $k$ and $b$. Given a block $B$, $|B|$ will represent the length of $B$.
\end{definition}

\begin{definition}\labd{1.2} Given an integer $b \geq 2$, the {\it $b$-ary expansion} of a real $x$ in $[0,1)$ is the (unique) expansion of the form
\begin{equation} \labeq{bary} 
x=\sum_{n=1}^{\infty} \frac {E_n} {b^n}=0.E_1 E_2 E_3 \ldots
\end{equation}
such that $E_n$ is in $\{0,1,\ldots,b-1\}$ for all $n$ with $E_n \neq b-1$ infinitely often.
\end{definition}

Denote by $N_n^b(B,x)$ the number of times a block $B$ occurs with its starting position no greater than $n$ in the $b$-ary expansion of $x$.

\begin{definition}\labd{1.3} A real number $x$ in $[0,1)$ is {\it normal in base $b$} if for all $k$ and blocks $B$ in base $b$ of length $k$, one has
\begin{equation} \labeq{bnormal1}
\lim_{n \rightarrow \infty} \frac {N_n^{b}(B,x)} {n}=b^{-k}.
\end{equation}
A number $x$ is {\it simply normal in base $b$} if \refeq{bnormal1} holds for $k=1$.
\end{definition}

 Borel introduced normal numbers in 1909 and proved that almost all (in the sense of Lebesgue measure) real numbers in $[0,1)$ are normal in all bases \cite{Borel}.  The best known example of a number that is normal in base $10$ is due to Champernowne \cite{Champernowne}.  The number
$$
H_{10}=0.1 \ 2 \ 3 \ 4 \ 5 \ 6 \ 7 \ 8 \ 9 \ 10  \ 11 \ 12 \ldots ,
$$
formed by concatenating the digits of every natural number written in increasing order in base $10$, is normal in base $10$.  Any $H_b$, formed similarly to $H_{10}$ but in base $b$, is known to be normal in base $b$. Since then, many examples have been given of numbers that are normal in at least one base.  One  can find a more thorough literature review in \cite{DT} and \cite{KuN}.

The $Q$-Cantor series expansion, first studied by Georg Cantor in \cite{Cantor}, is a natural generalization of the $b$-ary expansion.

\begin{definition}\labd{1.4} $Q=(q_n)_{n=1}^{\infty}$ is a {\it basic sequence} if each $q_n$ is an integer greater than or equal to $2$.
\end{definition}

\begin{definition}\labd{1.5} Given a basic sequence $Q$, the {\it $Q$-Cantor series expansion} of a real $x$ in $[0,1)$ is the (unique)\footnote{Uniqueness can be proven in the same way as for the $b$-ary expansion.} expansion of the form
\begin{equation} \labeq{cseries}
x=\sum_{n=1}^{\infty} \frac {E_n} {q_1 q_2 \ldots q_n}
\end{equation}
such that $E_n$ is in $\{0,1,\ldots,q_n-1\}$ for all $n$ with $E_n \neq q_n-1$ infinitely often. We abbreviate \refeq{cseries} with the notation $x=0.E_1E_2E_3\ldots$ w.r.t. $Q$.
\end{definition}

Clearly, the $b$-ary expansion is a special case of \refeq{cseries} where $q_n=b$ for all $n$.  If one thinks of a $b$-ary expansion as representing an outcome of repeatedly rolling a fair $b$-sided die, then a $Q$-Cantor series expansion may be thought of as representing an outcome of rolling a fair $q_1$ sided die, followed by a fair $q_2$ sided die and so on.  For example, if $q_n=n+1$ for all $n$, then the $Q$-Cantor series expansion of $e-2$ is
$$
e-2=\frac{1} {2}+\frac{1} {2 \cdot 3}+\frac{1} {2 \cdot 3 \cdot 4}+\ldots
$$
If $q_n=10$ for all $n$, then the $Q$-Cantor series expansion of $1/4$ is
$$
\frac {1} {4}=\frac{2} {10}+\frac {5} {10^2}+\frac {0} {10^3}+\frac {0} {10^4}+\ldots
$$

For a given basic sequence $Q$, let $N_n^Q(B,x)$ denote the number of times a block $B$ occurs starting at a position no greater than $n$ in the $Q$-Cantor series expansion of $x$. Additionally, define
$$
Q_n^{(k)}=\sum_{j=1}^n \frac {1} {q_j q_{j+1} \ldots q_{j+k-1}}.
$$

A. R\'enyi \cite{Renyi} defined a real number $x$ to be normal with respect to $Q$ if for all blocks $B$ of length $1$,
\begin{equation}\labeq{rnormal}
\lim_{n \rightarrow \infty} \frac {N_n^Q (B,x)} {Q_n^{(1)}}=1.
\end{equation}
If $q_n=b$ for all $n$, then \refeq{rnormal} is equivalent to {\it simple normality in base $b$}, but not equivalent to {\it normality in base $b$}.  Thus, we want to generalize normality in a way that is equivalent to normality in base $b$ when all $q_n=b$.

\begin{definition}\labd{1.7} A real number $x$ is {\it $Q$-normal of order $k$} if for all blocks $B$ of length $k$,
$$
\lim_{n \rightarrow \infty} \frac {N_n^Q (B,x)} {Q_n^{(k)}}=1.
$$
We say that $x$ is {\it $Q$-normal} if it is $Q$-normal of order $k$ for all $k$.  Additionally, $x$ is {\it simply $Q$-normal} if it is $Q$-normal of order $1$.
\end{definition}

We make the following definitions:

\begin{definition}\labd{1.8} A basic sequence $Q$ is {\it $k$-divergent} if
$$
\lim_{n \rightarrow \infty} Q_n^{(k)}=\infty.
$$
$Q$ is {\it fully divergent} if $Q$ is $k$-divergent for all $k$ and {\it $k$-convergent} if it is not $k$-divergent.
\end{definition}

\begin{definition}\labd{1.6} A basic sequence $Q$ is {\it infinite in limit} if $q_n \rightarrow \infty$.
\end{definition}

For $Q$ that are infinite in limit,
it has been shown that the set of all $x$ in $[0,1)$ that are $Q$-normal of order $k$ has full Lebesgue measure if and only if $Q$ is $k$-divergent \cite{Renyi}.  Therefore if $Q$ is infinite in limit, then the set of all $x$ in $[0,1)$ that are $Q$-normal has full Lebesgue measure if and only if $Q$ is fully divergent.  



\begin{definition}
Let~$x$ be a number in $[0,1)$ and let~$Q$ be a basic sequence.
Then $T_{Q,n}(x)$ is defined as
$$q_1\cdots q_n x\pmod{1}.$$
\end{definition}

\begin{definition}
A number~$x$ in $[0,1)$ is {\it $Q$-distribution normal} if
the sequence $(T_{Q,n}(x))_{n=0}^\infty$ is uniformly distributed in $[0,1)$.
\end{definition}

\begin{remark} For every basic sequence $Q$, the set of $Q$-distribution normal numbers has full Lebesgue measure.
\end{remark}

Note that in base~$b$, where $q_n=b$ for all $n$,
 the notions of $Q$-normality and
$Q$-distribution normality are equivalent. This equivalence
is fundamental in the study of normality
in base $b$. It is surprising that this
equivalence breaks down in the more general context of $Q$-Cantor series
for general $Q$. Examples are given in \cite{AlMa} of numbers that satisfy one notion of normality and not others.

In general, it is more difficult to give explicit constructions
of normal numbers (for various notions of normality) than it is to give
typicality results. An explicit construction of a basic sequence $Q$ and a real number $x$ such that $x$ is $Q$-normal and $Q$-distribution normal is given in  \cite{AlMa} and \cite{Mance}. In this paper, we will construct a set of $Q$-distribution normal numbers for any $Q$ that is infinite in limit.  None of these numbers will be $Q$-normal.  Additionally, this set of $Q$-distribution normal numbers will be perfect and nowhere dense.

We recall the following standard definition that will be useful in studying distribution normality:

\begin{definition}\labd{stardiscrepancy}
For a finite sequence $z=(z_1,\ldots,z_n)$, we define the
{\it star discrepancy $D_n^*=D_n^*(z_1,\ldots,z_n)$} as
$$
\sup_{0<\gamma\le 1}\left|{A([0,\gamma),z)\over n}-\gamma\right|.
$$
Given an infinite sequence $w=(w_1,w_2,\ldots)$, we define
$$
D_n^*(w)=D_n^*(w_1,w_2,\ldots,w_n).
$$
For convenience, set $D^*(z_1,\ldots,z_n)=D_n^*(z_1,\ldots,z_n)$.
\end{definition}

The star discrepancy will be useful to us due to the following theorem:

\begin{thrm}
The sequence $w=(w_1,w_2,\ldots)$ is uniformly distributed mod $1$ if and only if $\lim_{n \to \infty} D_n^*(w)=0$.
\end{thrm}

\begin{remark}
For any sequence $w$, $\frac {1} {n} \leq D_n^*(w) \leq 1$.
\end{remark}

The following theorem\footnote{T. \u Sal\'at proved a stronger result in \cite{Salat}, but we will not need it in this paper.} was proven by N. Korobov in \cite{Korobov} and will be of central importance in this paper:

\begin{thrm}\labt{Korobov}
Given a basic sequence $Q$ and a real number $x$ with
$Q$-Cantor series expansion
$x=\formalsum$; if $Q$ is infinite in limit, then $x$ is $Q$-distribution normal if and only if
$$
\left( {E_n\over q_n}\right)_{n=1}^\infty
$$
is uniformly distributed mod $1$.
\end{thrm}

We note the following theorem of J. Galambos \cite{Galambos2}:

\begin{thrm}\labt{galambos}
 Let $Q$ be a $1$-divergent basic sequence.  Let $E_k$ be the digits of the $Q$-Cantor series expansion of $x$ and put
$\theta_k=\theta_k(x)=E_k/q_k$.  Then, for almost all $x$ in $[0,1)$,
$$
D_n^*(\theta) \geq  \frac {1} {2n} \sum_{k=1}^n \frac {1} {q_k}
$$
for sufficiently large $n$.
\end{thrm}

A discrepancy estimate, valid for certain $Q$, will be given for the $Q$-distribution normal numbers that we will construct.
We will make use of the following definition from \cite{KuN}:

\begin{definition}\labd{almostarithmetic}
For $0 \leq \delta < 1$ and $\e>0$, a finite sequence $x_1<x_2<\cdots<x_N$ in $[0,1)$ is called an almost-arithmetic progression-$(\delta,\e)$ if there exists an $\eta$, $0<\eta \leq \e$, such that the following conditions are satisfied:
\begin{equation}\labeq{61}
0 \leq x_1 \leq \eta+\delta\eta;
\end{equation}
\begin{equation}\labeq{62}
\eta-\delta\eta \leq x_{n+1}-x_n \leq \eta+\delta\eta \hbox{ for } 1 \leq n \leq N-1;
\end{equation}
\begin{equation}\labeq{63}
1-\eta-\delta\eta \leq x_N <1.
\end{equation}
\end{definition}

Almost arithmetic progressions were introduced by P. O'Neil in \cite{ONeil}.  He proved that a sequence $(x_n)_n$ of real numbers in $[0,1)$ is uniformly distributed mod $1$ if and only if the following holds: for any three positive real numbers $\delta$, $\e$, and $\e'$, there exists a positive integer $N$ such that for all $n > N$, the initial segment $x_1,x_2,\ldots,x_n$ can be decomposed into an almost-arithmetic progression-$(\delta,\e)$ with at most $N_0$ elements left over, where $N_0 < \e' N$.

In \cite{AKS}, R. Adler, M.Keane, and M. Smorodinsky showed that the real number whose continued fraction expansion is given by the concatenation of the digits of the continued fraction expansion of the rational numbers 
\begin{equation}\labeq{AKS}
\frac {1} {2},\frac {1} {3}, \frac {2} {3}, \frac {1} {4}, \frac {2} {4}, \frac {3} {4}, \frac {1} {5}, \frac {2} {5}, \frac {3} {5}, \frac {4} {5},\ldots
\end{equation}
is normal with respect to the continued fraction expansion.  For every $Q$ that are infinite in limit, we use \refd{almostarithmetic} to construct a set $\Theta_Q$ of $Q$-distribution normal numbers that are defined similarly to the concatenation of the numbers in \refeq{AKS}. We prove the following results on $\Theta_Q$:

\begin{enumerate}
\item If $x \in \Theta_Q$, then $x$ is $Q$-distribution normal and not simply $Q$-normal (\reft{boundf3maintheorem} and \refp{fuckernotnormal}).
\item $\Theta_Q$ is perfect and nowhere dense (\reft{thetaperfect} and \reft{thetanowheredense}).
\item If $x \in \Theta_Q$, $x=0.E_1E_2\ldots$ w.r.t. $Q$, and $X=(E_n/q_n)_{n=1}^{\infty}$, then for certain basic sequences $Q$, there exists a constant $\gamma_Q$ such that for all $\psi>1$
$$
D_n^*(X) < \psi \cdot \gamma_Q \cdot n^{-1/2},
$$
for large enough $n$ (\reft{bdiscr3}).  For many basic sequences, we can determine the constant $\gamma_Q$.  In particular, $\gamma_Q=\sqrt{8}$ if $q_n \geq 5n$ for all $n$.
\item The Hausdorff dimension of $\Theta_Q$ is evaluated or approximated for several classes of basic sequences (\reft{hdtq1}, \reft{hdtq2}, \reft{hdtq3}, and \reft{polyhd}).  Given any $\al \in [0,1]$, we provide an example of a basic sequence $Q$ such that $\Theta_Q$ has $\al$ as its Hausdorff dimension (\reft{hdtq4}).
\end{enumerate}

\section{The construction}

For the rest of this section, we fix a basic sequence $Q$ that is infinite in limit.  

\subsection{Notation and conventions}
For the rest of this paper, let $\tau(n)=1+2+\ldots+n=\frac {n(n+1)}{2}$ be the $n^{th}$ triangular number.
Given a basic sequence $Q$, we will construct a sequence $l_1,l_2,l_3, \ldots$ of positive integers.  The following definition will be needed:

\begin{definition}
For each positive integer $j$, we define
$$
\nu_j=\min \{N : q_m \geq 2j^2 \hbox{ for all } m \geq N \}.
$$
\end{definition}

We now recursively define the sequence $l_1,l_2,l_3,\ldots$:

\begin{definition}
We set
$$
l_1=\max(\nu_2-1,1).
$$
Given $l_1,l_2,\ldots,l_{i-1}$, we define $l_i$ to be the smallest positive integer such that
$$
l_1+2l_2+3l_3+\ldots+il_i \geq \nu_{i+1}-1.
$$
Thus, we have
$$
l_i=\max(\min \{k : l_1+2l_2+\ldots+(i-1)l_{i-1}+ik \geq \nu_{i+1}-1\},1).
$$
\end{definition}

Additionally, for any non-negative integer $i$, we set
$$
L_i=\sum_{j=1}^i jl_j=l_1+2l_2+\ldots+il_i.
$$

\begin{lem} \labl{boundf3isF}
 Suppose that $a$, $c$, and $q$ are positive integers such that $q \geq 2a^2$.  Then there exists at least two integers $F$ such that
\begin{equation}\labeq{mesh1}
\frac {F} {q} \in \left[ \frac {c} {a}-\frac {1} {2a^2}, \frac {c} {a}+\frac {1} {2a^2}  \right].
\end{equation}
\end{lem}

\begin{proof}
We assume, for contradiction, that there are fewer than two solutions to \refeq{mesh1}.  Thus, there exists an integer $F$ such that
$$
\frac {F} {q}<\frac {c} {a}-\frac {1} {2a^2} \hbox{ and}
$$
$$
\frac {c} {a}+\frac {1} {2a^2}<\frac {F+2} {q}, \hbox{ so}
$$
\begin{equation}\labeq{mesh4}
\left[ \frac {c} {a}-\frac {1} {2a^2}, \frac {c} {a}+\frac {1} {2a^2}  \right] \nsubseteq \left[ \frac {F} {q}, \frac {F+2} {q} \right].
\end{equation}
By \refeq{mesh4}, we conclude that
$$
\left( \frac {c} {a}+\frac {1} {2a^2}\right) - \left(\frac {c} {a}-\frac {1} {2a^2} \right) <\frac {F+2} {q} - \frac {F} {q} \hbox{, so}
$$
\begin{equation}\labeq{mesh6}
\frac {1} {a^2} < \frac {2} {q}.
\end{equation}
Cross multiplying \refeq{mesh6} gives $q<2a^2$, which contradicts $q \geq 2a^2$.
\end{proof}

\begin{definition}
Let
$S_Q=\left\{ (a,b,c) \in \mathbb{N}^3 : b \leq l_a, c \leq a \right\}$
and define $\phi_Q:S_Q \to \mathbb{N}$ by
$\phi_Q(a,b,c)=L_{a-1}+(b-1)a+c$.
\end{definition}

\begin{lem}\labl{boundf3unique}
The function $\phi_Q$ is a bijection from $S_Q$ to $\mathbb{N}$.
\end{lem}

\begin{proof}
Starting at $n=1$, put $l_1$ boxes of length $1$, followed by $l_2$ boxes of length $2$, $l_3$ boxes of length $3$, and so on.  Then the position of component $c$ of the $b^{th}$ box of length $a$ is at  
$$
1l_1+2l_2+\ldots+(a-1)l_{a-1}+(b-1)a+c=\phi_Q(a,b,c),
$$
so $\phi_Q$ is a bijection from $S_Q$ to $\mathbb{N}$.
\end{proof}

\begin{definition}
The sequence $F=\left( F_{(a,b,c)} \right)_{(a,b,c) \in S_Q}$ is a {\it $Q$-special sequence} if $F_{(a,b,1)}=0$ for $(a,b,1) \in S_Q$ and
$$
\frac {F_{(a,b,c)}} {q_{\phi_Q(a,b,c)}} \in \left[ \frac {c-1} {a}-\frac {1} {2a^2}, \frac {c-1} {a}+\frac {1} {2a^2}  \right]
$$
for $(a,b,c) \in S_Q$ with $c>1$. Let $\Gamma_Q$ denote the set of all $Q$-special sequences.  
\end{definition}

Given a $Q$-special sequence $F$, \refl{boundf3unique} allows us to define $E_F=\left( E_{F,n} \right)_{n=1}^{\infty}$ as follows: 

\begin{definition}
Suppose that $F$ is a $Q$-special sequence.  For any positive integer $n$, we define 
$E_{F,n}=F_{\phi_Q^{-1}(n)}$ and let $E_F=\left( E_{F,n} \right)_{n=1}^{\infty}$.
\end{definition}

Given finite sequences $w_1,w_2,\ldots$, we let $w_1w_2w_3\ldots$ denote the concatenation of the sequences $w_1,w_2,\ldots$.

\begin{definition}
If $F$ is a $Q$-special sequence and $(a,b,1) \in S_Q$, then we define
$$
y_{F,a,b}=\left(  \frac {F_{(a,b,c)}} {q_{\phi_Q(a,b,c)}}  \right)_{c=1}^a
$$
and let $D_{F,a,b}^*=D^*(y_{F,a,b})$.
We also set
$$
y_F=y_{F,1,1}y_{F,1,2}\ldots y_{F,1,l_1}y_{F,2,1}y_{F,2,2}\ldots y_{F,2,l_2}y_{F,3,1}y_{F,3,2}\ldots y_{F,3,l_3}y_{F,4,1}\ldots.
$$
\end{definition}

\begin{definition}
If $F$ is a $Q$-special sequence, define
$$
x_F=\Fformalsum.
$$
We also let
$\Theta_Q= \{ x_F : F \in \Gamma_Q \}$.
\end{definition}

\begin{remark}
By construction,
$$
y_F=\left( \frac {E_{F,n}} {q_n} \right)_{n=1}^{\infty},
$$
so by \reft{Korobov}, $x_F$ is $Q$-distribution normal if and only if $y_F$ is uniformly distributed mod $1$.
\end{remark}

\subsection{Basic Lemmas}
We will use the following theorem from \cite{Niederreiter}:

\begin{thrm}\labt{AAdisc}
Let $x_1<x_2<\cdots<x_N$ be an almost arithmetic progression-$(\delta,\e)$ and let $\eta$ be the positive real number corresponding to the sequence according to \refd{almostarithmetic}.  Then
$$
D_N^* \leq \frac {1} {N}+\frac {\delta} {1+\sqrt{1-\delta^2}} \hbox{ for } \delta>0 \hbox{ and } D_N^* \leq \min \left( \eta, \frac {1} {N} \right) \hbox{ for } \delta=0.
$$
\end{thrm}

\begin{cor}\labc{AAdiscc}
Let $x_1<x_2<\cdots<x_N$ be an almost arithmetic progression-$(\delta,\e)$ and let $\eta$ be the positive real number corresponding to the sequence according to \refd{almostarithmetic}.  Then
$D_N^* \leq \frac {1} {N}+\delta$.
\end{cor}

\begin{lem}\labl{yabAA}
If $F$ is a $Q$-special sequence, then the sequence $y_{F,a,b}$ is an almost arithmetic progression-$\left( \frac {1} {a}, \frac {1} {a} \right)$ and
$D_{F,a,b}^* \leq \frac {2} {a}$.
\end{lem}

\begin{proof}
The case $a=1$ is trivial, so suppose that $a>1$.  To show that $y_{F,a,b}$ is an almost arithmetic progression-$\left( \frac {1} {a}, \frac {1} {a} \right)$, we first note that since $F_{(a,b,1)}=0$,
$$
0 \leq \frac {F_{(a,b,1)}} {q_{\phi_Q(a,b,1)}} \leq \frac {1} {a}+\frac {1} {a^2},
$$
so \refeq{61} holds.

Next, suppose that $2 \leq c \leq a-1$. By construction, 
$$
\frac {F_{(a,b,c)}} {q_{\phi_Q(a,b,c)}} \in \left[ \frac {c-1} {a}-\frac {1} {2a^2}, \frac {c-1} {a}+\frac {1} {2a^2}  \right]  \hbox{ and}
$$
$$
\frac {F_{(a,b,c+1)}} {q_{\phi_Q(a,b,c+1)}} \in \left[ \frac {c} {a}-\frac {1} {2a^2}, \frac {c} {a}+\frac {1} {2a^2}  \right]  \hbox{, so}
$$
\begin{equation}\labeq{boundf3upper}
\frac {F_{(a,b,c+1)}} {q_{\phi_Q(a,b,c+1)}}-\frac {F_{(a,b,c)}} {q_{\phi_Q(a,b,c)}} \leq \left( \frac {c} {a}+\frac {1} {2a^2} \right) - \left( \frac {c-1} {a}-\frac {1} {2a^2} \right) \hbox{ and}
\end{equation}
\begin{equation}\labeq{boundf3lower}
\frac {F_{(a,b,c+1)}} {q_{\phi_Q(a,b,c+1)}}-\frac {F_{(a,b,c)}} {q_{\phi_Q(a,b,c)}} \geq \left( \frac {c} {a}-\frac {1} {2a^2} \right) - \left( \frac {c-1} {a}+\frac {1} {2a^2} \right).
\end{equation}
Combining \refeq{boundf3upper} and \refeq{boundf3lower}, we see that
$$
\frac {1} {a}-\frac {1} {a^2} \leq \frac {F_{(a,b,c+1)}} {q_{\phi_Q(a,b,c+1)}}-\frac {F_{(a,b,c)}} {q_{\phi_Q(a,b,c)}} \leq \frac {1} {a}+\frac {1} {a^2},
$$
so \refeq{62} holds.

Lastly, by construction,
$$
\frac {a-1} {a}- \frac  {1} {a^2} \leq \frac {F_{(a,b,a)}} {q_{\phi_Q(a,b,a)}} < \frac {a-1} {a} + \frac {1} {a^2} \hbox{, so}
$$
$$
1-\frac {1} {a} -\frac {1} {a^2} \leq \frac {F_{(a,b,a)}} {q_{\phi_Q(a,b,a)}} \leq 1-\frac {1} {a}+\frac {1} {a^2} < 1
$$
and we have verified \refeq{63}.  Therefore, $y_{F,a,b}$ is an almost arithmetic progression-$\left( \frac {1} {a}, \frac {1} {a} \right)$.  By \refc{AAdiscc},
$$
D^*_{F,a,b} \leq \frac {1} {a} + \frac {1} {a} = \frac {2} {a}.
$$
\end{proof}

Throughout the rest of this paper,
for a given $n$, the letter $i=i(n)$ is the unique integer satisfying
$$
L_i < n \leq L_{i+1}.
$$
Given a positive integer~$n$, let 
$$
m=n-L_i.
$$
Note that $m$ can be written uniquely as 
$$
m=\al (i+1)+\be
$$
with
$$
 0\le\al\le l_{i+1} \hbox{ and } 0\le \be< i+1.
$$
We define $\alpha$ and $\beta$ as the unique integers satisfying these conditions.

The following results from \cite{KuN} will be needed:

\begin{lem}
If~$t$ is a positive integer and for $1\le j\le t$, 
$z_j$ is a finite sequence in $[0,1)$ with
star discrepancy at most~$\e_j$, then
$$
D^*(z_1 z_2 \cdots z_t)\le {\sum_{j=1}^t |z_j| \e_j \over\sum_{j=1}^t |z_j|}.
$$
\end{lem}

\begin{cor}\labc{kn2}
If~$t$ is a positive integer and for $1\le j\le t$,
$z_j$ is a finite sequence in $[0,1)$ with
star discrepancy at most~$\e_j$, then
$$
D^*(l_1 z_1 \cdots l_t z_t)\le {\sum_{j=1}^t l_j |z_j| \e_j \over\sum_{j=1}^t l_j |z_j|}.
$$
\end{cor}

Recall that $D^*(z)$ is bounded above by $1$ for all finite sequences $z$ of real numbers in $[0,1)$.
By \refc{kn2},
$$
D_n^*(y_F) \leq
f_i(\al,\be):=
\frac {\left(\sum_{j=1}^i l_j \cdot j \cdot \frac {2} {j} \right)+\al \cdot (i+1) \cdot \frac {2} {i+1}+\be} {\left( \sum_{j=1}^i jl_j \right)+(i+1)\al+\be}
=\frac { \left( \sum_{j=1}^i 2l_j \right) +2 \al+ \be} {\left( \sum_{j=1}^i jl_j \right)+(i+1)\al+\be}.
$$
Note that $f_i(\al,\be)$ is a rational function of  $\al$ and $\be$.
We consider the domain of $f_i$ to be
$\mathbb{R}_0^+ \times \mathbb{R}_0^+$, where $\mathbb{R}_0^+$ is the set
of all non-negative real numbers.
Given a $Q$-special sequence $F$, we now give an upper bound for $D_n^*(y_F)$.
Since $D_n^*(y_F)$ is at most $f_i(\al,\be)$, it is enough to bound
$f_i(\al,\be)$ from above on $[0,l_{i+1}]\times [0,i]$.

\begin{lem}\labl{boundf3}
If $i>2$,
\begin{equation}\labeq{boundf3condition}
\sum_{j=1}^i jl_j>\sum_{j=1}^i 2l_j \hbox{, and}
\end{equation}
$$
(w,z)\in \{0,\ldots,l_{i+1}\}\times\{0,\ldots,i\},
$$
then
$$
f_i(w,z)<f_i(0,i+1)=\frac {\left(\sum_{j=1}^i 2l_j \right)+i+1} {\left(\sum_{j=1}^i jl_j \right)+i+1}.
$$
\end{lem}

\begin{proof}
To bound~$f_i(w,z)$, we first compute its partial derivatives
$\fsz (w,z)$ and $\fsw (w,z)$.  We will show that $\fsw (w,z)$ is always negative and $\fsz (w,z)$ is always positive. Note that this is enough to prove \refl{boundf3}
 since $w \ge 0$ and $z<i+1$.

First, we note that $f_i (w,z)$ is a rational function of $w$ and $z$ of the form
$$
f_i (w,z)=\frac {C+Dw+Ez} {F+Gw+Hz},
$$
where
\begin{equation}\labeq{boundf3coeffs}
C=\sum_{j=1}^i 2l_j, \ D=2, \ E=1, \ F=\sum_{j=1}^i jl_j, \ G=i+1 \hbox{, and } H=1.
\end{equation}
Therefore,
$$
\frac {\partial f_i} {\partial w} (w,z)=\frac {D(F+Hz)-G(C+Ez)} {(F+Gw+Hz)^2} \hbox{ and}
$$
$$
\frac {\partial f_i} {\partial z} (w,z)=\frac {E(F+Gw)-H(C+Dw)} {(F+Gw+Hz)^2}.
$$
Thus, the sign of $\frac {\partial f_i} {\partial w} (w,z)$ does not depend on $w$ and the sign of $\frac {\partial f_i} {\partial z} (w,z)$ does not depend on $z$.
We will show that $f_i(w,z)$ is a
decreasing function of $w$ by proving that
\begin{equation}\labeq{boundf3one}
D(F+Hz)<G(C+Ez).
\end{equation}
Similarly, we show that $f_i (w,z)$ is an increasing function of $z$ by verifying that
\begin{equation}\labeq{boundf3two}
E(F+Gw)>H(C+Dw).
\end{equation}
Substituting the values in \refeq{boundf3coeffs} into \refeq{boundf3one}, we need to prove that
\begin{equation}\labeq{655}
2\left( \sum_{j=1}^i jl_j +z \right) < (i+1)\left( \sum_{j=1}^i 2l_j+z \right).
\end{equation}
Distributing both sides, we see that \refeq{655}  is equivalent to
$$
\left( 2\sum_{j=1}^i jl_j \right) +2z < 2\left( (i+1)\sum_{j=1}^i l_j \right) +(i+1)z.
$$
However, $i>2$ by hypothesis and
$$
\sum_{j=1}^i jl_j < (i+1)\sum_{j=1}^i l_j,
$$
so \refeq{boundf3one} holds.

By substituting the values in \refeq{boundf3coeffs} into \refeq{boundf3two}, we need to prove that
$$
\left( \sum_{j=1}^i jl_j \right)+(i+1)w >\left( \sum_{j=1}^i 2l_j \right)+2w.
$$
By \refeq{boundf3condition}, we know that
$
\sum_{j=1}^i jl_j>\sum_{j=1}^i 2l_j
$
and $i>2$, so \refeq{boundf3two} holds.
\end{proof}

Set
$$
\bar\e_i=
f_i(0,i+1)=\frac {\left(\sum_{j=1}^i 2l_j \right)+i+1} {\left(\sum_{j=1}^i jl_j \right)+i+1}.
$$
We will now prove a series of lemmas to show that $\bar\e_i \to 0$.  The following was proven by O. Toeplitz in \cite{Toeplitz}:
\begin{thrm}\labt{Toeplitz}
Let $(\gamma_{n,k} : 1 \leq k \leq n, n \geq 1)$ be an array of real numbers such that:
\begin{enumerate}
\item $ \lim_{n \rightarrow \infty} \gamma_{n,k}=0$ for each $k \in \mathbb{N}$;
\item $\lim_{n \rightarrow \infty} \sum_{k=1}^n \gamma_{n,k}=1$;
\item there exists $C>0$ such that for all positive integers $n$: $\sum_{k=1}^n | \gamma_{n,k} | \leq C$.
\end{enumerate}
Then for any convergent sequence $( \alpha_n )$, the transformed sequence $( \beta_n )$ given by
$$
\beta_n=\sum_{k=1}^n \gamma_{n,k} \alpha_k, n \geq 1,
$$
is also convergent and 
$$
\lim_{n \rightarrow \infty} \beta_n =\lim_{n \rightarrow \infty} \alpha_n.
$$
\end{thrm}

We will need the following result that follows from \reft{Toeplitz}:

\begin{lem}\labl{tcorr}
Let $L$ be a real number and $(a_n)_{n=1}^{\infty}$ and $(b_n)_{n=1}^{\infty}$ be two sequences of positive real numbers such that
$$
\sum_{n=1}^{\infty} b_n=\infty \hbox{ and } \lim_{n \to \infty} \frac {a_n} {b_n}=L.
$$
Then
$$
\lim_{n \to \infty} \frac {a_1+a_2+\ldots+a_n} {b_1+b_2+\ldots+b_n}=L.
$$
\end{lem}

\begin{proof}
Let $\alpha_n=\frac {a_n} {b_n}$ and put
$$
\gamma_{n,k}=\frac {b_k} {b_1+b_2+\ldots+b_n}.
$$
We now verify that $\{\gamma_{n,k}\}$ satisfies the hypothesis of \reft{Toeplitz}.  Clearly,
$\lim_{n \to \infty} \gamma_{n,k}=0$
for all $k$, as $\sum_{n=1}^{\infty} b_n=\infty$.  Next, we note that
$$
\sum_{k=1}^n \gamma_{n,k}=\sum_{k=1}^n \frac {b_k} {b_1+b_2+\ldots+b_n}=\frac {b_1+b_2+\ldots+b_n} {b_1+b_2+\ldots+b_n}=1,
$$
so the second condition of \reft{Toeplitz} is satisfied.  The third condition is trivially satisfied for $C=1$ as $\gamma_{n,k}>0$ for all $n$ and $k$.

Thus, if
$$
\beta_n=\sum_{k=1}^n \gamma_{n,k} \alpha_k=\sum_{k=1}^n  \frac {a_k} {b_1+b_2+\ldots+b_n}=\frac {a_1+a_2+\ldots+a_n} {b_1+b_2+\ldots+b_n},
$$
then by \reft{Toeplitz}, we see that
$$
\lim_{n \to \infty} \frac {a_1+a_2+\ldots+a_n} {b_1+b_2+\ldots+b_n}=\lim_{n \to \infty} \frac {a_n} {b_n}=L.
$$
\end{proof}

We may now show that $\bar\e_i \to 0$.

\begin{lem}\labl{boundf3tozero}
$\lim_{n \to \infty} \bar\e_i=0$.
\end{lem}

\begin{proof}
We will first show that $\lim_{i \to \infty} \bar\e_i \to 0$.  The lemma will then follow as $i=i(n)$ satisfies $\lim_{n \to \infty} i(n)=\infty$.

We apply \refl{tcorr} with $a_1=2l_1+2$, $b_1=l_1+2$ and for $j>1$,
$a_j=2l_j+1$, and  $b_j=jl_j+1$.
Thus,
$$
a_1+a_2+\ldots+a_i=\left(\sum_{j=1}^i 2l_j \right)+i+1  \hbox{ and}
$$
$$
b_1+b_2+\ldots+b_i=\left(\sum_{j=1}^i jl_j \right)+i+1.
$$
Since
$
\lim_{i \to \infty} \frac {a_i} {b_i}=\lim_{i \to \infty} \frac {2l_i+1} {il_i+1}=0,
$
we see that
$$
\lim_{i \to \infty} \bar\e_i=\lim_{i \to \infty} \frac {\left(\sum_{j=1}^i 2l_j \right)+i+1} {\left(\sum_{j=1}^i jl_j \right)+i+1}=\lim_{i \to \infty} \frac {a_i} {b_i}=0.
$$
\end{proof}

\subsection{Main Theorem}

\begin{thrm}\labt{boundf3maintheorem}
Suppose that $F$ is a $Q$-special sequence.  Then $x_F$ is $Q$-distribution normal.
\end{thrm}

\begin{proof}
Suppose that $n$ is large enough so that $i>2$ and
$(i-2)l_i> l_1$.
Then
\begin{equation}\labeq{lafer1}
il_i+2l_2+l_1>2l_i+2l_2+2l_1.
\end{equation}
We also note that 
\begin{equation}\labeq{lafer2}
jl_j>2l_j \hbox{ for } j>2.
\end{equation}
Combining \refeq{lafer1} and \refeq{lafer2},
$$
\sum_{j=1}^i jl_j>\sum_{j=1}^i 2l_j.
$$
By \refl{boundf3}, $D_n^*(y_F) < \bar\e_{i(n)}$ and by \refl{boundf3tozero}, $\bar\e_{i(n)} \to 0$, so the sequence $y_F$ is uniformly distributed mod $1$.  Thus, by \reft{Korobov}, $x_F$ is $Q$-distribution normal.
\end{proof}

We will now show that while \reft{boundf3maintheorem} allows us to construct $Q$-distribution normal numbers, none of these numbers will be simply $Q$-normal.

\begin{prop}\labp{fuckernotnormal}
If $F$ is a $Q$-special sequence, then $x_F$ is not simply $Q$-normal.
\end{prop}

\begin{proof}
If $Q$ is $1$-convergent, then $x_F$ is not simply $Q$-normal as the digit $0$ occurs infinitely often in the $Q$-Cantor series expansion of $x_F$.

Next, suppose that $Q$ is $1$-divergent.  We will show that the digit $1$ may only occur finitely often in the $Q$-Cantor series expansion of $x_F$.  Suppose that $(a,b,2) \in F$ and $a \geq 2$.  Then, by construction, we have
$$
\frac {F_{(a,b,2)}} {q_{\phi_Q(a,b,2)}} \in \left[ \frac {1} {a}-\frac {1} {2a^2}, \frac {1} {a}+\frac {1} {2a^2}  \right]
$$
and $q_{\phi_Q(a,b,2)} \geq 2a^2$.  Thus, we see that
$$
\frac {F_{(a,b,2)}} {q_{\phi_Q(a,b,2)}} \geq \frac {1} {a}-\frac {1} {2a^2} \hbox{, so}
$$
\begin{equation}\labeq{bigvagina}
F_{(a,b,2)} \geq \left(\frac {1} {a}-\frac {1} {2a^2} \right)q_{\phi_Q(a,b,2)} \geq \left(\frac {1} {a}-\frac {1} {2a^2} \right) \cdot 2a^2=2a-1>1.
\end{equation}
Thus, by \refeq{bigvagina}, $F_{(a,b,2)} > 1$ when $a \geq 2$.  Since $F_{(a,b,1)}=0$ whenever $(a,b,1) \in S_Q$, there are at most finitely many $n$ such that $E_{F,n}=1$, so $x_F$ is not simply $Q$-normal.
\end{proof}

\section{Other properties of $\Theta_Q$}
\subsection{Discrepancy Results}
\begin{lem}\labl{bdiscr1}
Suppose that $Q$ is a basic sequence such that there exists constants $M$ and $t$ with $\nu_{i+1}-\nu_i \leq Mi$ for all $i > t$.  Then $l_i \leq \ceil{M+1}$ for all $i > t$.
\end{lem}

\begin{proof}
Suppose that $i>t$ and $l_i \geq 2$.  Then by definition of the sequence $(l_i)_i$, we have the inequalities\footnote{Note that we cannot conclude that $L_i < \nu_{i+1}+i-1$ if $l_i=1$.  To see this, consider the basic sequence given by $q_n=8^n$, where $l_i=1$ , $L_i=\frac {i(i+1)} {2}$, and $\nu_i=\ceil{\log_8 (2i^2)}$ for all $i$.}
$$
\nu_{i+1}-1 \leq L_i < \nu_{i+1}+i-1 \hbox{ and } \nu_i-1 \leq L_{i-1}.
$$
Thus,
$$
L_i=L_{i-1}+il_i< \nu_{i+1}+i-1 \hbox{, so}
$$
$$
l_i<\frac {\nu_{i+1}+i-1-L_{i-1}} {i} \leq \frac {\nu_{i+1}+i-1-(\nu_i-1)} {i} 
$$
$$
= 1+\frac {\nu_{i+1}-\nu_i} {i} \leq 1+\frac {Mi} {i}=1+M \leq \ceil{1+M}.
$$
\end{proof}

\begin{prop}\labp{bdiscr2}
Suppose that $Q$ is a basic sequence such that there exists constants $M$ and $t$ where $l_j \leq M$ for $j > t$.  Then for all $Q$-special sequences $F$ and real numbers $\psi>1$, we have
$$
D_n^*(y_F) < \psi \cdot \sqrt{2M} \cdot (2M+1) \cdot n^{-1/2},
$$
for large enough $n$.
\end{prop}

\begin{proof}
By \refl{boundf3}, for large enough $n$, we have
$$
D_n^*(y_F) < \frac {\left(\sum_{j=1}^i 2l_j \right)+i+1} {\left(\sum_{j=1}^i jl_j \right)+i+1}.
$$
Set $\kappa=\sum_{j=1}^t jl_j$.  Since $l_j \geq 1$ for all $j$, we see that
$$
D_n^*(y_F) < \frac {2\kappa+\left(\sum_{j=1}^i 2M \right)+i+1} {\left(\sum_{j=1}^i j \cdot 1 \right)+i+1}=
\frac {2\kappa+2Mi+i+1} {\frac {i(i+1)} {2}+i+1}=\frac {(2\kappa+1)+(2M+1)i} {\frac {i^2+3i+3} {2}}
$$
\begin{equation}\labeq{balzak1}
< \frac {2(2\kappa+1)+2(2M+1)i} {i^2+\frac {3} {2}i}=\frac {2(2\kappa+1)/i+2(2M+1)} {i+\frac {3} {2}}.
\end{equation}
However, 
$$
\frac {i(i+1)} {2}=\sum_{j=1}^i j \cdot 1 \leq \sum_{j=1}^i jl_j <  n \leq \kappa+\sum_{j=1}^{i+1} jl_j \leq \kappa+\sum_{j=1}^{i+1} jM=\kappa+\frac {(i+1)(i+2)} {2} M.
$$
Thus, we see that $i \geq p$, where $p$ is the positive solution to $n=\kappa+\frac {(p+1)(p+2)} {2} M$.  Therefore,
\begin{equation}\labeq{balzak2}
p=\frac {-3+\sqrt{\frac {8} {M} \cdot n+\left(1-\frac {8\kappa} {M} \right)}} {2}.
\end{equation}
Substituting \refeq{balzak2} into \refeq{balzak1}, we arrive at the inequality
$$
D_n^*(y_F)< \frac {2(\kappa+1)/i+2(2M+1)} {\left(\frac {-3+\sqrt{\frac {8} {M} \cdot n+\left(1-\frac {8\kappa} {M} \right)}} {2}+\frac {3} {2}\right)}
=\frac {4(\kappa+1)/i+4(2M+1)} { \sqrt{\frac {8} {M} \cdot n+\left(1-\frac {8\kappa} {M} \right)}}.
$$
Let $\psi>1$.  Then for large enough $n$,
$$
\frac {4(\kappa+1)/i(n)+4(2M+1)} { \sqrt{\frac {8} {M} \cdot n+\left(1-\frac {8\kappa} {M} \right)}} < \psi \cdot \frac {4(2M+1)} {\sqrt{\frac {8} {M} \cdot n}}=\psi \cdot \sqrt{2M} \cdot (2M+1) \cdot n^{-1/2},
$$
so $D_n^*(y_F) < \psi \cdot \sqrt{2M} \cdot (2M+1) \cdot n^{-1/2}$.
\end{proof}

\begin{thrm}\labt{bdiscr3}
Suppose that $Q$ is a basic sequence such that there exists constants $M$ and $t$ where $\nu_{i+1}-\nu_i \leq Mi$ for $j > t$.  Then for all $Q$-special sequences $F$ and real numbers $\psi>1$, we have
$$
D_n^*(y_F) < \psi \cdot \sqrt{2\ceil{M+1}} \cdot (2\ceil{M+1}+1) \cdot n^{-1/2},
$$
for large enough $n$.
\end{thrm}

\begin{proof}
This follows directly from \refl{bdiscr1} and \refp{bdiscr2}.
\end{proof}
\begin{remark}
If $q_m \geq 2n^2$ for $\tau(n-1) < m \leq \tau(n)$, then $l_i=1$ for all $i$ and \reft{bdiscr3} implies that for all $\psi>1$ and large enough $n$, we have
\begin{equation}\labeq{bdiscrr}
D_n^*(y_F)< \psi \cdot \sqrt{8} \cdot n^{-1/2}.
\end{equation}
For example, \refeq{bdiscrr} holds if $q_n \geq 5n$ for all $n$.
\end{remark}

\subsection{$\Theta_Q$ is perfect and nowhere dense}
The goal of this subsection will be to show that $\Theta_Q$ is a perfect, nowhere dense subset of $[0,1)$. First, we first remark that the {\it existance} of a set of normal numbers that is perfect and nowhere dense should not be surprising.  However, constructing a specific example of such a set may not lend itself to an obvious solution.  

We will now work towards showing that $\Theta_Q$ is perfect and nowhere dense.  In order to proceed, we define a function, $d$, from $\Gamma_Q \times \Gamma_Q$ to $\mathbb{R}$:

\begin{definition}
Suppose that $F_1$ and $F_2$ are $Q$-special sequences.  If $F_1 \neq F_2$, we define
$$
\zeta_{F_1,F_2}=\min \{n : E_{F_1,n} \neq E_{F_2,n} \}.
$$
Define\footnote{$(\Gamma_Q,d)$ is a  metric space.} $d: \Gamma_Q \times \Gamma_Q \to \mathbb{R}$ by
$$
d(F_1,F_2)=\left\{ \begin{array}{ll}
\frac {1} {q_1 q_2 \ldots q_{\zeta_{F_1,F_2}-1}} & \textrm{if $F_1 \neq F_2$}\\
0		& \textrm{if $F_1=F_2$}
\end{array} \right. .
$$
\end{definition}

\begin{lem}
If $F_1,F_2 \in \Gamma_Q$, then
$\left| x_{F_1}-x_{F_2} \right| \leq d(F_1,F_2)$.
\end{lem}

\begin{proof}
Let $n=\zeta_{F_1,F_2}$.  We write the $Q$-Cantor series expansions of $x_{F_1}$ and $x_{F_2}$ as follows:
$$
x_{F_1}=\frac {E_1} {q_1}+ \frac {E_2} {q_1 q_2}+\ldots+\frac {E_{n-1}} {q_1q_2 \cdots q_{n-1}}+\frac {E_{F_1,n}} {\qn}+\frac {E_{F_1,n+1}} {q_1q_2 \cdots q_{n+1}}+\ldots \hbox{ and}
$$
$$
x_{F_2}=\frac {E_1} {q_1}+ \frac {E_2} {q_1 q_2}+\ldots+\frac {E_{n-1}} {q_1q_2 \cdots q_{n-1}}+\frac {E_{F_2,n}} {\qn}+\frac {E_{F_2,n+1}} {q_1q_2 \cdots q_{n+1}}+\ldots \hbox{, so}
$$
$$
\left| x_{F_1}-x_{F_2} \right| = \left| \left( \frac {E_{F_1,n}} {q_1q_2 \cdots q_{n-1}} - \frac {E_{F_2,n}} {q_1q_2 \cdots q_{n-1}} \right) + \left( \frac {E_{F_1,n+1}} {q_1q_2 \cdots q_{n+1}}-\frac {E_{F_2,n+1}} {q_1q_2 \cdots q_{n+1}} \right) +\ldots \right|
$$
$$
\leq \frac {|E_{F_1,n}-E_{F_2,n}|} {\qn}+ \frac {|E_{F_1,n+1}-E_{F_2,n+1}|} {q_1q_2 \cdots q_{n+1}}+\ldots \leq \frac {1} {q_1q_2 \cdots q_{n-1}}=d(F_1,F_2).
$$
\end{proof}

\begin{lem}\labl{specialisperfect}
If $F \in \Gamma_Q$, then there exists a sequence of $Q$-special sequences $F_1,F_2,F_3,\ldots$ such that $F \neq F_n$ for all $n$ and
$\lim_{n \to \infty} d(F,F_n)=0$.
\end{lem}

\begin{proof}
By \refl{boundf3isF}, we may define a sequence of $Q$-special sequences as follows.  Let $n$ be any positive integer and put $(\al,\be,\ga)=\phi_Q^{-1}(n)$.  We must now consider three cases.  First, if $\ga \neq 1$, then for $m \neq n$, we set $E_{n,m}=E_{F,m}$ and we let $E_{n,n} \neq E_{F,n}$ be any value that satisfies
$$
\frac {E_{n,n}} {q_n} \in \left[ \frac {\ga-1} {\al}-\frac {1} {2\al^2}, \frac {\ga-1} {\al}+\frac {1} {2\al^2}  \right].
$$

Second, we suppose that $\ga=1$ and $\al>1$.  Put $(\al',\be',\ga')=\phi_Q^{-1}(n+1)$. Then for $m \neq n+1$, we set $E_{n,m}=E_{F,m}$ and we let $E_{n,n+1} \neq E_{F,n+1}$ be any value that satisfies
$$
\frac {E_{n,n+1}} {q_{n+1}} \in \left[ \frac {\ga'-1} {\al'}-\frac {1} {2\al'^2}, \frac {\ga'-1} {\al'}+\frac {1} {2\al'^2}  \right].
$$

Third, we consider the case where $\al=\ga=1$.  Set $t=\phi_Q(2,1,2)$ and note that $t>n$.  Then for $m \neq t$, put $E_{n,m}=E_{F,m}$ and let $E_{n,t} \neq E_{F,t}$ be any value that satisfies
$$
\frac {E_{n,t}} {q_t} \in \left[ \frac {2-1} {2}-\frac {1} {2 \cdot 2^2}, \frac {2-1} {2}+\frac {1} {2 \cdot 2^2}  \right]=\left[\frac {3} {8}, \frac {5} {8} \right].
$$

Now that we have determined the sequence $(E_{n,m})_{m=1}^{\infty}$, set $F_n=(E_{n,\phi_Q(a,b,c)})_{(a,b,c) \in S_Q}$. Thus, $F \neq F_n$ for all $n$ and for large enough $m$, we have
$$
d(F,F_m) \leq \max \left(\frac {1} {\qm}, \frac {1} {q_1q_2 \cdots q_{m-1}}      \right)=\frac {1} {q_1q_2 \cdots q_{m-1}},
$$
so $F_n \to F$.
\end{proof}

\begin{thrm}\labt{thetaperfect}
The set $\Theta_Q$ is perfect.
\end{thrm}

\begin{proof}
Suppose that $x \in \Theta_Q$ and that $x=x_F$.  By \refl{specialisperfect}, there exists a sequence of $Q$-special sequences $F_1,F_2,F_3,\ldots$, none of which are equal to $F$, with $F_n \to F$.  Thus, $x \neq x_{F_n}$ for all $n$.  Let $\e >0$ and suppose that $N$ is large enough so that for all $n > N$, we have $d(F,F_n) < \e$.  Clearly, $|x-x_{F_n}| \leq d(F,F_n) < \e$, so $x_{F_n} \to x_F$ and $\Theta_Q$ is perfect.
\end{proof}

\begin{lem}\labl{nowheredenselemma1}
If $a \geq 1$, then
\begin{equation}\labeq{nowheredenselemma11}
\frac {a-1} {a}+\frac {1.5} {2a^2} < 1.
\end{equation}
\end{lem}

\begin{proof}
We rewrite \refeq{nowheredenselemma11} as
\begin{equation}\labeq{nowheredenselemma12}
\frac {2a^2-2a+1.5} {2a^2} < 1.
\end{equation}
Thus, to verify \refeq{nowheredenselemma12}, we need show that
\begin{equation}\labeq{nowheredenselemma13}
2a^2-2a+1.5<2a^2.
\end{equation}
However, as $a \geq 1$, we see that  $-2a+1.5<0$, so \refeq{nowheredenselemma13} follows.
\end{proof}

\begin{thrm}\labt{thetanowheredense}
The set $\Theta_Q$ is nowhere dense.
\end{thrm}

\begin{proof}
Let $I \subset [0,1)$ be any interval such that $\Theta_Q \cap I \neq \emptyset$.  We will show that there exists an interval $K \subset I$ such that $\Theta_Q \cap K=\emptyset$.  Thus, there exists a positive integer $n$ and an interval $J \subset I$ with
$$
J=\left[ \frac {E_1} {q_1}+\frac {E_2} {q_1 q_2}+\ldots+\frac {E_n} {\qn}, \frac {E_1} {q_1}+\frac {E_2} {q_1 q_2}+\ldots+\frac {E_n+1} {\qn} \right)
$$
and $E_j \in [0,q_j-1) \cap \mathbb{Z}$ for $j=1,2,\ldots,n$.  Put $(a,b,c)=\phi_Q^{-1}(n+1)$. By \refl{nowheredenselemma1}, we may set 
$$
K=\Bigg[ \frac {E_1} {q_1}+\ldots+\frac {E_n} {\qn}+\left( \frac {a-1} {a}+\frac {1.5} {2a^2}  \right)\frac {1} {\qn}, \frac {E_1} {q_1}+\ldots+\frac {E_n+1} {\qn} \Bigg).
$$
If $\Theta_Q \cap J=\emptyset$, we are finished, so assume that $\Theta_Q \cap J \neq \emptyset$.  Suppose that $F \in \Gamma_Q$ is such that $x_F \in J$ and
$$
x=0.E_1E_2 \ldots E_n E_{n+1} E_{n+2} \ldots \hbox{ w.r.t. } Q.
$$
By construction, if $c \neq 1$, we have
$$
\frac {E_{n+1}} {q_{n+1}} \in \left[ \frac {c-1} {a}-\frac {1} {2a^2}, \frac {c-1} {a}+\frac {1} {2a^2}  \right].
$$
If $c=1$, then $E_{n+1}=0$.  Therefore,
$$
x_F \leq \frac {E_1} {q_1}+\frac {E_2} {q_1 q_2}+\ldots+\frac {E_n} {\qn}+\left( \frac {c-1} {a}+\frac {1} {2a^2}  \right)\frac {1} {\qn}
$$
$$
<\frac {E_1} {q_1}+\frac {E_2} {q_1 q_2}+\ldots+\frac {E_n} {\qn}+\left( \frac {a-1} {a}+\frac {1.5} {2a^2}  \right)\frac {1} {\qn},
$$
so $x_F \notin K$.  Hence, $K \cap \Theta_Q= \emptyset$ and $\Theta_Q$ is nowhere dense.
\end{proof}

\subsection{Hausdorff dimension of $\Theta_Q$}

Given a basic sequence $Q$ and a positive integer $n$, we will define the functions $a(n), b(n)$, and $c(n)$ by $(a(n),b(n),c(n))=\phi_Q^{-1}(n)$.  Set $\om_n=\#\{E_{F,n} : F \in \Gamma_Q\}$,
$$
A(k)=\left\{ \begin{array}{ll}
1 & \textrm{if $1 \leq k \leq l_1$}\\
p & \textrm{if $l_1+l_2+\ldots+l_{p-1} < k \leq l_1+l_2+\ldots+l_p$}
\end{array} \right. ,
$$
and $\ga(k)=A(1)+A(2)+\ldots+A(k)$.

Note that $\om_n=1$ if and only if $c(n)=1$.  By \refl{boundf3isF}, we are guaranteed that $\om_n \geq 2$ if $c(n) \neq 1$.  Additionally, we can say that
\begin{equation}\labeq{hdr1}
\frac {q_n} {a(n)^2} \leq \om_n \leq \frac {q_n} {a(n)^2}+1<\frac {2q_n} {a(n)^2}
\end{equation}
when $\om_n \neq 1$. If $q_n$ grows quickly enough that $l_1=l_2=\ldots=1$, then $A(k)=k$ and $\ga(k)=\tau(k)$, so $a(n)=\floor{(1+\sqrt{8n-7})/2}$.  Thus, we have the inequality
\begin{equation}\labeq{hdr2}
\sqrt{n} \leq a(n) < \sqrt{3n}.
\end{equation}
Combining \refeq{hdr1} and \refeq{hdr2}, we see that
\begin{equation}\labeq{ombound}
\frac {q_n} {3n} \leq \om_n  < \frac {2q_n} {n}.
\end{equation}

\begin{definition}\footnote{A basic sequence may still grow slowly no matter how fast $q_n$ grows when $n$ is restricted to those values for which $\om_n=1$.}
A basic sequence $Q$ {\it grows nicely} if $n^s=o(q_n)$ for all positive integers $s$,
$$
\log q_{\tau(k-1)+1} + \log q_{\tau(k)}=o\left( \log \prod_{n=1}^{\tau(k-1)-1} q_n  \right) \hbox{, and}
$$
$$
\log \prod_{n=0}^{k-2} q_{\tau(k)+1}=o\left( \log \prod_{n=1}^{\tau(k-1)-1} q_n  \right).
$$
A basic sequence $Q$ {\it grows slowly} if there exists a constant $M$ such that $\om_n \leq M$ for all $n \geq 1$.  Lastly, $Q$ {\it grows  quickly} if
$$
\log \prod_{n=1}^{\tau(k)-1} q_n=o(\log q_{\tau(k)})
$$
\end{definition}

\begin{example}\labe{exofgrowths}
The basic sequences given by $q_n=n+1$ and $q_n=\max(2,\floor{\log n})$ grow slowly.  If $t \geq 2$, then $q_n=\floor{t^n}$ and $q_n=2^{2^n}$  give  examples of nicely growing basic sequences.  If we let $q_1=2$ and $q_{n+1}=2^{\qn}$, then $Q$ grows  quickly.
\end{example}

If $J \subset [0,1)$ is a subset of $[0,1)$, we will denote its Hausdorff dimension by $\dimh J$.
In this section, we will compute the Hausdorff dimension of $\Theta_Q$ for a few classes of basic sequences.  We will show that $\dimht=0$ when $Q$ grows slowly or  quickly.  When $Q$ grows nicely, we will have $\dimht=1$.

\begin{definition}
Let $J$ be any non-empty subset of $[0,1)$ and let $C_\delta(J)$ be the  smallest number of sets of diameter at most $\delta$ which can cover $J$.  Then the {\it box-counting dimension} of $J$, if it exists, is defined as
$$
\dimb{J} =\lim_{\delta \to 0} \frac {\log C_\delta(J)} {-\log \delta}.
$$
The {\it lower box-counting dimension} and {\it upper box-counting dimension} of $J$ are defined as
$$
\ldimb{J}=\liminf_{\delta \to 0} \frac {\log C_\delta(J)} {-\log \delta}\hbox{ and}
$$
$$
\udimb{J}=\limsup_{\delta \to 0} \frac {\log C_\delta(J)} {-\log \delta},
$$
respectively.
\end{definition}

The following standard result will be used frequently and without mention:\footnote{See \cite{Falconer}.}

\begin{thrm}\labt{hdlessbox}
Let $J$ be a non-empty subset of $[0,1]$.  Then
$$0 \leq \dimh J \leq \ldimb J \leq \udimb J \leq 1.$$
\end{thrm}

We will make use of the following general construction found in \cite{Falconer}.  Suppose that $[0,1]=I_0 \supset I_1 \supset I_2\supset \ldots$ is a decreasing sequence of sets, with each $I_k$ a union of a finite number of disjoint closed intervals (called {\it $k^{th}$ level basic intervals}).  Then we will consider the set $\cap_{k=0}^{\infty} I_k$.  We will construct a set $\tq'$ that may be written in this form such that $\dimht=\dimhtp$. 

Given a block of digits $B=(b_1,b_2,\ldots,b_s)$ and a positive integer $n$, define 
$$
\mathscr{S}_{Q,B}=\{x=0.E_1E_2\ldots\hbox{ w.r.t }Q : E_1=b_1, \ldots, E_{t}=b_s\}.
$$
Let $P_n$  be the set of all possible values of $E_n(x)$ for $x \in \tq$. Put $J_0=[0,1)$ and
$$
J_k=\bigcup_{B \in \prod_{n=1}^{\ga(k)} P_n} \mathscr{S}_{Q,B}
$$
Then $J_k \subset J_{k-1}$ for all $k \geq 0$ and $\tq=\cap_{k=0}^{\infty} J_k$, which gives the following:

\begin{prop}
$\tq$ can be written in the form $\cap_{k=0}^{\infty} J_k$, where each $J_k$ is the union of a finite number of disjoint half-open intervals.
\end{prop}

We now set $I_k=\overline{J_k}$ for all $k \geq 0$ and put $\tqp=\cap_{k=0}^{\infty} I_k$.  Since each set $J_k$ consists of only a finite number of intervals, the set $I_k \backslash J_k$ is finite.

\begin{lem}
$\dimht=\dimhtp$.
\end{lem}

\begin{proof}
The lemma follows as $\tqp \backslash \tq$ is a countable set.
\end{proof}

For $k \geq 1$, we note that, by construction, there are $\om_1\om_2 \cdots \om_{\ga(k)-1}$ $k^{th}$ level intervals and they are all of length $(q_1q_2 \cdots q_{\ga(k)})^{-1}$.  Additionally, they are all separated by a distance of at least $(q_1q_2 \cdots q_{\ga(k)})^{-1}(1+2/A(k)^2)$.  This gives us the following:

\begin{equation}\labeq{boxdim}
\ldimbt=\liminf_{k \to \infty} \frac {\log \left( \om_1 \om_2 \cdots \om_{\ga(k)-1} \right)} {\log \left(q_1q_2 \cdots q_{\ga(k)} \right)},
\end{equation}
$$
\udimbt=\limsup_{k \to \infty} \frac {\log \left( \om_1 \om_2 \cdots \om_{\ga(k)-1} \right)} {\log \left(q_1q_2 \cdots q_{\ga(k)} \right)}, \hbox{ and}
$$
$$
\dimbt=\lim_{k \to \infty} \frac {\log \left( \om_1 \om_2 \cdots \om_{\ga(k)-1} \right)} {\log \left(q_1q_2 \cdots q_{\ga(k)} \right)}.
$$

\begin{thrm}\labt{hdtq1}
Suppose that $Q$ grows slowly.  Then $$\dimht=\dimbt=0.$$
\end{thrm}
\begin{proof}
Since $Q$ is infinite in limit, for all $z>M$, there exists a positive integer $t$ such that $q_1q_2 \cdots q_{\ga(k)} \geq z^{\ga(k)}$ for all $k>t$.  Substituting $\om_n \leq M$ into \refeq{boxdim}, we see that
$$
\udimbt \leq \limsup_{k \to \infty} \frac {\log M^{\ga(k)-1}} {\log z^{\ga(k)}} = \frac {\log M} {\log z},
$$
so $\udimbt=0$.
\end{proof}

We will use the following from \cite{Falconer}:

\begin{thrm}\labt{hdlower}
Suppose that each $(k-1)^{th}$ level interval of $I_{k-1}$ contains at least $m_k$ $k^{th}$ level intervals ($k=1,2,\ldots$) which are separated by gaps of at least $\e_k$, where $0 \le \e_{k+1} < \e_k$ for each $k$.  Then
$$
\dimh \left( \bigcap_{k=0}^{\infty} I_k \right) \geq \liminf_{k \to \infty} \frac {\log (m_1m_2 \cdots m_{k-1})} {-\log (m_k\e_k)}.
$$
\end{thrm}

\begin{lem}\labl{firstlowerhd}
$$
\dimhtp \geq \liminf_{k \to \infty} \frac {\log \left( \om_1 \om_2 \cdots \om_{\ga(k-1)-1} \right)} {\log \left( \frac {q_1 q_2 \cdots q_{\ga(k)}} {\om_{\ga(k-1)} \om_{\ga(k-1)+1} \cdots \om_{\ga(k)-1}} \right)}.
$$
\end{lem}

\begin{proof}
We substitute $m_1 \cdots m_{k-1}=\om_1 \om_2 \cdots \om_{\ga(k-1)-1}$, $m_k=\om_{\ga(k-1)} \om_{\ga(k-1)+1} \cdots \om_{\ga(k)-1}$, and $\e_k=(q_1q_2 \cdots q_{\ga(k)})^{-1}(1+2/A(k)^2)$ into \reft{hdlower}.
Since $\lim_{k \to \infty} A(k)=\infty$, we see that $1<1+2/A(k)^2\leq 3$, so
$$
\liminf_{k \to \infty} \frac {\log (m_1m_2 \cdots m_{k-1})} {-\log (m_k\e_k)}=
\liminf_{k \to \infty} \frac {\log \left( \om_1 \om_2 \cdots \om_{\ga(k-1)-1} \right)} {\log \left( \frac {q_1 q_2 \cdots q_{\ga(k)}} {\om_{\ga(k-1)} \om_{\ga(k-1)+1} \cdots \om_{\ga(k)-1}} \cdot (1+2/A(k))^{-1}\right)}
$$
$$
=\liminf_{k \to \infty} \frac {\log \left( \om_1 \om_2 \cdots \om_{\ga(k-1)-1} \right)} {\log \left( \frac {q_1 q_2 \cdots q_{\ga(k)}} {\om_{\ga(k-1)} \om_{\ga(k-1)+1} \cdots \om_{\ga(k)-1}} \right)}.
$$
\end{proof}

\begin{lem}\labl{secondlowerhd}
Suppose that $l_i=1$ for all $i$.  Then
\begin{equation}\labeq{estofhd}
\dimhtp \geq \liminf_{k \to \infty} \frac {\log \left(\prod_{n=1}^{\tau(k-1)-1} q_n\right)     -\log 3^{\tau(k-2)-1}-\log (\tau(k-1)-1)! -\log \left(\prod_{n=0}^{k-2} q_{\tau(n)+1} \right)} {\log \left( \prod_{n=1}^{\tau(k-1)-1} q_n \right)+\log q_{\tau(k-1)+1} +\log q_{\tau(k)}+\log 3^{k-1}+\log \left( \frac {(\tau(k)-1)!} {(\tau(k-1)-1)!\cdot (\tau(k-1)+1)}\right)}
\end{equation}
\end{lem}

\begin{proof}
Since $l_i=1$ for all $i$, \refeq{ombound} holds and $\ga(k)=\tau(k)$ for all $k$.  Note that $\om_n = 1$ if and only if $n=\tau(k)+1$ for some $k$.  Therefore,
$$
\om_1 \om_2 \cdots \om_{\ga(k-1)-1} \geq \frac {q_1} {3 \cdot 1} \cdot \frac {q_2} {3 \cdot 2} \cdots \frac {q_{\tau(k-1)-1}} {3 \cdot (\tau(k-1)-1)} \cdot \prod_{n=0}^{k-2} \frac {3(\tau(n)+1)} {q_{\tau(n)+1}}
$$
$$
\geq \left( \prod_{n=1}^{\tau(k-1)-1} q_n \right) \cdot 3^{-(\tau(k-1)-1)} \cdot (\tau(k-1)-1)!^{-1} \cdot 3^{k-1} \cdot \prod_{n=0}^{k-2} q_{\tau(n)+1}^{-1}
$$
$$
=\left(\prod_{n=1}^{\tau(k-1)-1} q_n\right)  \cdot 3^{-(\tau(k-2)-1)}\cdot (\tau(k-1)-1)!^{-1} \prod_{n=0}^{k-2} q_{\tau(n)+1}, \hbox{ so}
$$
$$
\log \left( \om_1 \om_2 \cdots \om_{\ga(k-1)-1} \right) \geq \log \left(\prod_{n=1}^{\tau(k-1)-1} q_n\right)-\log 3^{\tau(k-2)-1}-\log (\tau(k-1)-1)!-\log \left(\prod_{n=0}^{k-2} q_{\tau(n)+1} \right).
$$
Next, since $\om_{\ga(k-1)+1}=1$, we arrive at the estimate
$$
\frac {q_1 q_2 \cdots q_{\ga(k)}} {\om_{\ga(k-1)} \om_{\ga(k-1)+1} \cdots \om_{\ga(k)-1}} \leq 
\frac {q_1 q_2 \cdots q_{\tau(k)}} {\left(\frac {q_{\tau(k-1)}} {3\tau(k-1)} \cdot \frac {q_{\tau(k-1)+1}} {3(\tau(k-1)+1)} \cdots \frac {q_{\tau(k)-1}} {3(\tau(k)-1)}  \right) \cdot \frac {3(\tau(k-1)+1)} {q_{\tau(k-1)+1}}  }
$$
$$
=\left( \prod_{n=1}^{\tau(k-1)-1} q_n \right)\cdot q_{\tau(k-1)+1} \cdot q_{\tau(k)} \cdot 3^{k-1} \cdot \frac {(\tau(k)-1)!} {(\tau(k-1)-1)!\cdot (\tau(k-1)+1)}.
$$
Thus, by \refl{firstlowerhd}
$$
\dimhtp \geq \liminf_{k \to \infty} \frac {\log \left(\prod_{n=1}^{\tau(k-1)-1} q_n\right)     -\log 3^{\tau(k-2)-1}-\log (\tau(k-1)-1)! -\log \left(\prod_{n=0}^{k-2} q_{\tau(n)+1} \right)} {\log \left( \prod_{n=1}^{\tau(k-1)-1} q_n \right)+\log q_{\tau(k-1)+1} +\log q_{\tau(k)}+\log 3^{k-1}+\log \left( \frac {(\tau(k)-1)!} {(\tau(k-1)-1)!\cdot (\tau(k-1)+1)}\right)}.
$$
\end{proof}

\begin{thrm}\labt{hdtq2}
Suppose that $Q$ grows nicely.  Then $\dimht=1$.
\end{thrm}

\begin{proof}
We will show that $\dimhtp=1$, so that $\dimht=1$ immediately follows.  We need only consider the case where $l_i=1$ for all $i$.
Since $Q$ grows nicely, the dominant term of both the numerator and denominator in \refeq{estofhd} is $\log  \prod_{n=1}^{\tau(k-1)-1} q_n$, so $\dimhtp=1$ by \refl{secondlowerhd}.
\end{proof}

\begin{thrm}\labt{hdtq3}
Suppose that $Q$ grows  quickly.  Then $\dimht=\dimbt=0$.
\end{thrm}

\begin{proof}
It will be sufficient to consider the case where $l_k=1$ for all $k$. We will show that $\udimbt=0$.  Recall that $\om_n< \frac {2q_n} {n}$, so $\ga(k)=\tau(k)$ and
$$
\om_1 \om_2 \cdots \om_{\ga(k)-1} < \frac {2q_1} {1} \cdot \frac {2q_2} {2} \cdots \frac {2q_{\tau(k)-1}} {\tau(k)-1}=\left(\prod_{n=1}^{\tau(k)-1} q_n \right) \cdot 2^{\tau(k)-1} / (\tau(k)-1)!,\hbox{ so}
$$
\begin{equation}\labeq{fastbox}
\udimbt \leq \limsup_{k \to \infty} \frac {\log \prod_{n=1}^{\tau(k)-1} q_n+\log 2^{\tau(k)-1}-\log (\tau(k)-1)!} {\log \prod_{n=1}^{\tau(k)-1} q_n+\log q_{\tau(k)} }
\end{equation}
However, the dominant terms in the numerator and denominator of \refeq{fastbox} are $\log \prod_{n=1}^{\tau(k)-1} q_n$ and $\log q_{\tau(k)}$, respectively, so
$$
\udimbt \leq \limsup_{k \to \infty} \frac {\log \prod_{n=1}^{\tau(k)-1} q_n} {\log q_{\tau(k)}}=0.
$$
\end{proof}

The Hausdorff dimension of $\Theta_Q$ is less certain when $q_n$ grows like a polynomial.  The following lemma will be needed:

\begin{lem}\labl{factoriallemma}
$$
\log \left(\frac {(\tau(k)-1)!} {(\tau(k-1)-1)!\cdot (\tau(k-1)+1)} \right) =o\left( \log (\tau(k-1)-1)! \right)
$$
\end{lem}

\begin{proof}
Suppose that $k>2$.  Then
$$
\frac {(\tau(k)-1)!} {(\tau(k-1)-1)!\cdot (\tau(k-1)+1)} < \frac {(\tau(k)-1)!} {(\tau(k-1)-1)!}
$$
$$
<(\tau(k)-1)^k=e^{k \log \left(\frac {1} {2} (k^2+k-2)\right)}<e^{k \log(k^2)}=e^{2k \log k},\hbox{ so}
$$
$$
\log \left(\frac {(\tau(k)-1)!} {(\tau(k-1)-1)!\cdot (\tau(k-1)+1)} \right) <2k\log k.
$$

By Stirling's formula,
$$
(\tau(k-1)-1)! > \sqrt{2\pi}(\tau(k-1)-1)^{\tau(k-1)-1/2}e^{-(\tau(k-1)-1)}
$$
$$
=\sqrt{2\pi} \left(\frac{1} {2} (k^2-k-2)  \right)^{\frac {1} {2}  \left( k^2-k-1\right)}e^{-\frac {1} {2}\left(k^2-k-2 \right)}
$$
$$
=\sqrt{2\pi}e^{\left( \frac{1} {2} (k^2-k-2) \log \left( \frac {1} {2} ( k^2-k-1) \right)-\frac {1} {2}\left(k^2-k-2 \right)\right)},\hbox{ so}
$$
$$
\log (\tau(k-1)-1)!>\frac {1} {2}(k^2-k-2)\left(  \log \left( \frac {1} {2} ( k^2-k-1) \right) -1 \right).
$$
Since $\lim_{k \to \infty} \frac {2k\log k} {\frac {1} {2}(k^2-k-2)\left(  \log \left( \frac {1} {2} ( k^2-k-1) \right) -1 \right)}=0$, the lemma follows.
\end{proof}

\begin{thrm}\labt{polyhd}
Suppose that there exists reals number $t > 1$ and  $\la_1,\la_2 \geq 1$ such that $\la_1  n^t \leq q_n \leq \la_2 n^t$ for all $n$ and $q_m \geq 2p^2$ for $\tau(p-1) < m \leq \tau(p)$.  Then $\dimbt=1$ and
$$
1-\frac {1} {t} \leq \dimht \leq 1.
$$
\end{thrm}
\begin{proof}
Since $\la_1 n^t \leq q_n \leq \la_2  n^t$,
$$
\log \la_1^{\tau(k-1)-1} + t \log (\tau(k-1)-1)! 
\leq \log \left(\prod_{n=1}^{\tau(k-1)-1} q_n\right) 
\leq \log \la_2^{\tau(k-1)-1} +t\log (\tau(k-1)-1)!,
$$
so
$$
\log \left(\prod_{n=1}^{\tau(k-1)-1} q_n\right)-\log (\tau(k-1)-1)! \geq (t-1) \log (\tau(k-1)-1)!+\log \la_1^{\tau(k-1)-1} .
$$
Note that $l_i=1$ for all $i$, so by \refl{secondlowerhd} and \refl{factoriallemma}
$$
\dimhtp \geq \liminf_{k \to \infty} \frac {(t-1)\log (\tau(k-1)-1)!+\log \la_1^{\tau(k-1)-1}} {t\log (\tau(k-1)-1)!+\log \la_2^{\tau(k-1)-1}}
$$
$$
=\liminf_{k \to \infty} \frac {(t-1) \cdot \log (\tau(k-1)-1)!} {t \cdot \log (\tau(k-1)-1)!}=1-\frac {1} {t}.
$$
A similar computation gives $\ldimbt=1$, so $\dimbt=1$.
\end{proof}

Let $\al \in (0,1)$.  We will now work towards constructing a basic sequence $Q_\al$ such that $\dimhtqa=\al$.
Define the basic sequence $Q_\al=(q_{\al,n})_n$ by
\begin{equation}
q_{\al,n}=\left\{ \begin{array}{ll}
\max\left(\floor{\left(\prod_{m=1}^{n-1} q_{\al,m}   \right)^{(1-\al)/\al}},2n^2 \right) & \textrm{if $n=\tau(k)$ for even $k$}\\
2n^2 & \textrm{for all other values of $n$}
\end{array} \right. .
\end{equation}
We will write $V_k=q_{\al,\tau(k)}$ and $P_k=\prod_{n=1}^{\tau(k)-1} q_{\al,n}$, so for  large enough integers $k$ that are even,
\begin{equation}\labeq{vk}
V_k=\floor{P_k^{(1-\al)/\al}}
\end{equation}

\begin{lem}\labl{Qident}
If $k$ is even, then 
$$
\frac {1-\al} {\al} \log P_{k-1} < \log V_k < \frac {1-\al} {\al} \log P_{k-1}+ \frac {4-4\al} {\al} k \log k.
$$
\end{lem}

\begin{proof}
$$
\log V_k \leq \log \left(P_{k-1} \cdot \prod_{n=\tau(k-1)}^{\tau(k)-1} 2n^2 \right)^{(1-\al)/\al}=
\frac {1-\al} {\al} \log P_{k-1}+\frac {1-\al} {\al} \log \prod_{n=\tau(k-1)}^{\tau(k)-1} 2n^2
$$
$$
< \frac {1-\al} {\al} \log P_{k-1}+\frac {1-\al} {\al} \log  \left(2\tau(k)^2\right)^k<\frac {1-\al} {\al} \log P_{k-1}+\frac {1-\al} {\al}k \log k^4
$$
$$
=\frac {1-\al} {\al} \log P_{k-1}+\frac {4-4\al} {\al} k \log k.
$$
The lower bound follows similarly.
\end{proof}

\begin{thrm}\labt{hdtq4}
If $\al \in (0,1)$, then $\dimhtqa=\ldimbtqa=\al$ and $\udimbtqa=1$.
\end{thrm}
\begin{proof}
For  this basic sequence, $l_i=1$ for all $i$, so we may use our usual estimates.  Thus, by \refeq{boxdim} and \refeq{vk}
$$
\ldimbtqa=\liminf_{k \to \infty} \frac {\log \left( \om_1 \om_2 \cdots \om_{\ga(k)-1} \right)} {\log \left(q_{\al,1}q_{\al,2} \cdots q_{\al,\ga(k)} \right)} \leq \liminf_{k \to \infty} \frac {\log \prod_{n=1}^{\tau(k)-1} \frac {2q_{\al,n}} {n}} {\log \prod_{n=1}^{\tau(k)-1}q_{\al,n}+\log q_{\al,\tau(n)}}
$$
$$
= \min\left(\lim_{k \to \infty, k\hbox{ even}} \frac {\log P_k} {\log P_k+\frac {1-\al} {\al} \log P_k}, \lim_{k \to \infty, k\hbox{ odd}} \frac {\log P_k} {\log P_k+\log (2\tau(k)^2)}\right).
$$
$$
=\min\left(\frac {1}{1+\frac {1-\al} {\al}},1 \right)=\al.
$$
Following a similar computation, $\udimbtqa=1$.
By \refl{secondlowerhd} and \refl{Qident}
$$
\dimhtqap \geq \liminf_{k \to \infty} \frac {\log P_{k-1}} {\log P_{k-1}+\log V_k}
$$
$$
=\min\left(\lim_{k \to \infty, k\hbox{ even}} \frac {\log P_{k-1}} {\log P_{k-1}+\frac {1-\al} {\al} \log P_{k-1}}, \lim_{k \to \infty, k\hbox{ odd}} \frac {\log P_{k-1}} {\log P_{k-1}+\log (2\tau(k)^2)}   \right)=\al,
$$
so $\dimhtqa=\ldimbtqa=\al$.

\end{proof}

\subsection*{Acknowledgements}
I would like to thank Vitaly Bergelson and Gerald Edgar for many helpful discussions and Alexandra Nichols for her help in editing this paper.


\begin{thebibliography}{HD}




\normalsize
\baselineskip=17pt


\bibitem{AKS} R. Adler, M. Keane and M. Smorodinsky, {\it A construction of a normal number for the continued fraction transformation}, J. Number Theory 13, 95--105 (1981)

\bibitem{AlMa} C. Altomare and B. Mance, {\it Cantor series constructions contrasting two notions of normality}, Monatsh. Math. 164, 1--22 (2011)

\bibitem{Borel}  E. Borel, {\it Les probabilit\'es d\'enombrables et leurs applications arithm\'etiques}, Rend. Circ. Mat. Palermo   27, 247--271 (1909)

\bibitem{Cantor}  G. Cantor, {\it \"Uber die einfachen Zahlensysteme}, Zeitschrift f\"ur Math. und Physik   14, 121--128 (1869)

\bibitem{Champernowne} D. G. Champernowne, {\it The construction of decimals normal in the scale of ten}, Journal of the London Mathematical Society, 8, 254--260 (1933)

\bibitem{DT} M. Drmota and R. F. Tichy, {\it Sequences, Discrepancies and Applications}, Springer-Verlag, Berlin Heidelberg, 1997

\bibitem{Falconer} Falconer, Kenneth J. (2003). Fractal geometry. Mathematical foundations and applications. John Wiley \& Sons, Inc., Hoboken, New Jersey

\bibitem{Galambos2}  J. Galambos, {\it Uniformly distributed sequences mod 1 and Cantor's series representation}, Czech. Math J.  26,  636-641 (1976)

\bibitem{Korobov}  N. Korobov,{\it Concerning some questions of uniform distribution modulo one}, Izv. Akad. Nauk SSSR Ser. Mat.  14,  215--238 (1950)

\bibitem{KuN} L. Kuipers and H. Niederreiter, {\it Uniform Distribution of Sequences}, Dover, Mineola, NY, 2006.

\bibitem{Lafer}  P. Lafer,  {\it Normal numbers with respect to Cantor series representation}. Thesis (Ph.D.)-Washington State University, 52pp (1974)

\bibitem{Mance}  B. Mance, {\it Construction of normal numbers with respect to the Q-Cantor series expansion for certain Q},  Acta Arith. 148, 135--152 (2011)

\bibitem{Niederreiter}  H. Niederreiter, {\it Almost-arithmetic progressions and uniform distribution}, Trans. Amer. Math. Soc.  161,  283--292 (1971)

\bibitem{ONeil}  P. E. O'Neil, {\it A new criterion for uniform distribution}, Proc. Amer. Math. Soc.  24,  1--5 (1970)

\bibitem{Renyi} A. R\'enyi, {\it On the distribution of the digits in Cantor's series}, Mat. Lapok, 7, 77--100 (1956)

\bibitem{Salat}  T. \u Sal\'at,  {\it Zu einigen Fragen der Gleichvertleilung (Mod 1)}, Czech. Math. J.  18 (93),  476--488  (1968)

\bibitem{Toeplitz}  O. Toeplitz, {\it \"Uber die lineare Mittelbildungen}, Prace mat.-fiz.  22,  113--118 (1911)

\end{thebibliography}
\end{document}